\documentclass[10pt,a4paper,english]{article}

\usepackage{amssymb,amsmath}
\usepackage{amsfonts}
\usepackage{eucal}
\usepackage{enumerate}
\usepackage{amsthm}
\usepackage{macros-perso}
\usepackage[all,dvips]{xy}
\CompileMatrices

\newtheorem{defn}{Definition}
\newtheorem{thm}{Theorem}
\newtheorem*{thm*}{Theorem}
\newtheorem{cor}{Corollary}
\newtheorem*{cor*}{Corollary}
\newtheorem{lemm}{Lemma}
\newtheorem{prop}{Proposition}

\newtheorem*{ex}{Example}

\renewcommand{\naa}{\nabla^{\alpha}}
\newcommand{\nab}{\nabla^{\beta}}
\def\TT{\R_{\infty}}

\newcommand{\norm}[1]{w(#1)}
\newcommand{\union}[1]{\cup \{#1\}}
\author{Eduardo Corel\footnote{E-mail : ecorel@gwdg.de}}
\title{G\'erard-Levelt membranes}
\date{}
\begin{document}

\maketitle


\begin{abstract}
We present an unexpected application of tropical convexity to the determination of invariants for linear systems of differential equations. We show that the classical G\'erard-Levelt lattice saturation procedure can be geometrically understood in terms of a projection on the tropical linear space attached to a subset of the local affine Bruhat-Tits building, that we call the {\it G\'erard-Levelt membrane}. 
This provides a way to compute the true Poincar\'e rank, but also the Katz rank of a meromorphic connection without having to perform gauge transforms nor ramifications of the variable. 
We finally present an efficient algorithm to compute this tropical projection map, generalising Ardila's method for the case of the Bergman fan to the case of the tight-span of a valuated matroid.
\end{abstract}


\tableofcontents



\section*{Introduction}

Given a meromorphic linear differential system on the Riemann sphere,
\begin{equation}
\frac{dX}{dz}=A(z)X\textrm{ with }A(z)\in \m\left(\C(z)\right),\label{sd}
\end{equation}
it is important to determine whether a
singularity of $A$ is a {\em regular} singular point for the system
(\ref{sd}). Unlike with scalar linear differential equations, for
which there is a purely algebraic condition on the orders of the poles
of the coefficients due to L. Fuchs \cite{Fu}, a system (\ref{sd}) can
display arbitrarily high pole orders at a regular singularity.

Consider the differential system, expanded in the neighbourhood of the singular
point (assumed for simplicity to be $z=0$) as follows,
where we put $\theta=z\frac{d}{dz}$, 
\begin{equation}
\theta X=\frac{1}{z^{p}}\sum_{i\gsl 0}A_iz^iX\textrm{ with }p\gsl 0
\textrm{ and }A_0\neq 0\textrm{ if }p>0.
\label{SF}
\end{equation}
The integer $p$ is traditionally known as the {\em Poincar\'e rank}
$\p(A)$ of the system. Finding the type of singularity involves
knowing the minimum value $m(A)$, sometimes known as the {\it true
  Poincar\'e rank}, of this rank under gauge transformations 
\begin{equation}
A_{[P]}=P^{-1}AP-P^{-1}\theta P\textrm{ with
}P=\sum_{k\gsl k_0}P_kz^k\in \G(K)\textrm{ where }K=\C((z)).\label{gauge}
\end{equation}

Several lines of research have been opened to tackle this problem. 
The most classical tries to iteratively construct a suitable gauge
transformation $P$, usually coefficient by coefficient in the series
expansion. Featured methods rely on the linear algebra over $\C$ involved by
equation (\ref{gauge}), like Moser and continuators
(\cite{Mo,HW}), whose methods are widely used nowadays in
computer algebra,  or other researchers such as \cite{Ju,BJL}, while \cite{BV} use Lie group
theoretic tools. 

The nature of a singularity of $A$ can also considered from point of
view of meromorphic connections \cite{De}, and especially, as a
question of stability of certain {\it lattices} under the differential
operator induced by the connection \cite{M,Ka}. We focus here
specifically on the approach of {\it saturating lattices} used by
G\'erard and Levelt \cite{GL}: the true Poincar\'e rank $m(\na)$ is
the minimum integer $k$ such that the sequence of $k$-saturated
lattices (recalled in section~\ref{GLs}) eventually stabilises.

Recent work has shown close relations between the geometric framework
of the Bruhat-Tits building of $\SL(K)$, for some discrete valued
field such as $K=\C((z))$, and tropical convexity \cite{KT,JSY,W}.
We show here that G\'erard and Levelt's approach can be formulated and
efficiently computed in this framework, as a projection on
the tropical linear space $L_p$ attached to the valuated matroid $p$
corresponding to a given {\it membrane} in the Bruhat-Tits building.

Results of Keel and Tevelev \cite{KT} show that, when lattices are in a same
membrane of the Bruhat-Tits building, they are {\it homothetic} if and
only if they are projected on the same point of a tropical linear
space by an explicit nearest point projection map (\cite{JSY}, see also
\cite{CGQ,GK}).
We need to be able to tell whether lattices are actually equal. We show
that the projection map respects the valuation, and 
projects only {\it equal} lattices on a same point (theorem \ref{KT}). 

We then construct the {\it G\'erard-Levelt membrane} that contains all
the $k$-saturated lattices (proposition~\ref{GLm}), and give a tropical version
of G\'erard-Levelt's lattice stabilisation criterion (corollary~\ref{P}). 
Third, we show that the same formula amazingly allows for the computation of a
more subtle invariant, namely the {\it Katz rank} of the
connection (theorem~\ref{K}), without having to either to compute a single gauge
transformation nor perform the usually required ramification of the
variable.

The projection map given in~\cite{JSY} has unfortunately a high complexity.
We give in section~\ref{proj} an efficient algorithm to compute this
projection map onto the linear tropical space attached to a valuated
matroid. We generalise the algorithmic approach to tropical projection developed by Ardila
\cite{A} for ordinary matroids, to the case of {\it valuated matroids}
defined by Dress and collaborators \cite{DW}. The algorithm presented here,
which computes the nearest $\ell_{\infty}$-projection on the
tight-span of a valuated matroid, has a wider applicability than the
differential computations explained in the previous parts, especially in
phylogenetics \cite{DT}.

\section*{Acknowledgements}

I have benefited from very interesting and stimulating discussions
with Federico Ardila, Josephine Yu, Michael Joswig, Annette Werner, Felipe Rinc\'on and
St\'ephane Gaubert, held at the Tropical Geometry Workshop at the CIEM
in Castro Urdiales (Spain) in December 2011.

\section{Meromorphic connections}

A {\em meromorphic connection} is a map $\appto{\na}{V\simeq
  K^n}{\om(V)=V\otimes_{K}\om^1_{\C}(K)}$ which is $\C$-linear and
satisfies the Leibniz rule $$\na(fv)=v\otimes df+f\na
v\text{ for }f \in K\text{ and }v \in V.$$ 
The {\em matrix $\mat(\na,(e))$} is given by $\na e_{j}=-\sum^n_{i=1}
e_i\otimes \Omega_{ij}$ for a basis $(e)$.  A basis change $P\in
\G(K)$ {\it gauge-transforms} the matrix of $\na$ by 
\begin{equation} \label{jauge} \Omega_{[P]}=P^{-1}\Omega P-P^{-1}dP.
\end{equation}
Contracting with $z^{k+1}\frac{d}{dz}$ yields a differential operator $\na_k$, and system (\ref{sd}) is
the expression of $\na_{-1}(v)=0$ in the basis $(e)$.

A \emph{lattice} $\la$ in $V$ is a free sub-$\h$-module
of rank $n$, that is a module of the form $$\la=\bigoplus_{i=1}^n\h
e_i\textrm{ for some basis }(e)\textrm{ of }V.$$

The \emph{Poincar\'e rank of $\na$ on the lattice
$\la$} is defined as the integer
\begin{equation}
\label{poin}
\p_{\la}(\na)=-v_{\la}(\la+\na_0(\la))=
\min\{k\in \N\tq \na_k(\la)\subset \la\}=\max_{i,j}(-v(\Omega_{ij})-1,0).
\end{equation}

The {\em true Poincar\'e rank} $m(\na)=\min_{\la} \p_{\la}(\na)$
characterises the nature of the singularity of~(\ref{sd}), in the
sense that $z=0$ is regular if and only if $m(\na)=0$.
 
\subsection{G\'erard-Levelt's saturated lattices}
\label{GLs}
For any vector $e\in V$ and any derivation $\tau\in \de$, define for
$\ell \in \N$ the family $$Z_{\tau}^{\ell}(e)=(e,\nap e,\ldots,\nap^{\ell} e).$$
The module $\h_{\tau}^{\ell}(e)$ induced over $\h$ by $Z_{\tau}^{\ell}(e)$ only depends on the
valuation $v(\tau)$ of the derivation $\tau$. We can therefore
restrict ourselves to the particular derivations
$\tau_k=z^{k+1}\ddz{}$ for $k \in \N$.
In this case, we put $\na_{\tau_k}=\na_k$, and note $Z^{\ell}_k(e)$
and $\h^{\ell}_k(e)$ the corresponding objects.
For $k\gsl 1$, G\'erard and Levelt define the
lattices $$F^{\ell}_k(\la)=\la+\na_k\la +\cdots +\na_k^{\ell}\la.$$ 
Note that $$F^{\ell}_k(\la)=\sum_{i=1}^n Z^{\ell}_k(e_i)\textrm{ for any basis
}(e_1,\ldots,e_n)\textrm{ of }\la.$$

\begin{thm}[G\'erard, Levelt]
The true Poincar\'e rank $m(\na)$ of $\na$ is
$$m(\na)=\min\{k\in\N \tq F^{n}_k(\la)=F^{n-1}_k(\la)\}\textrm{ for any
lattice }\la\subset V.$$
\end{thm}

This means that $k\gsl m(\na)$ if and only if
$\na_k(F^{n-1}_k(\la))\subset F^{n-1}_k(\la)$ for some (equivalently,
any) lattice $\la$ in $V$, that is, that the Poincar\'e rank on
$F^{n-1}_k(\la)$ is at most $k$. Stated otherwise, finding the true
Poincar\'e rank is finding the largest lattice whose Poincar\'e rank
is bounded by its index in the following sequence
\begin{equation}
F^{n-1}_0(\la)\supset \cdots \supset F^{n-1}_{p-1}(\la) \supset \la.
\end{equation}

Let us extend this notation to multi-indices. Let $\ell\gsl
0$, and let $\alpha=(\alpha_1,\ldots,\alpha_{\ell})\in \Z^{\ell}$ be an integer
multi-index of length $\abs{\alpha}=\ell$ and {\em
weight} $\norm{\alpha}=\alpha_1+\cdots +\alpha_{\ell}$. Let us define also
the partial multi-indices $\alpha_{\vert j}=(\alpha_1,\ldots,\alpha_j)$
and $$\naa=\na_{\alpha_{\ell}}\circ \cdots \circ \na_{\alpha_1}.$$
Let by convention $\alpha_{\vert 0}=\epsilon$ and
$\na^{\epsilon}=\idd_V$ for the empty sequence $\epsilon$. 
Let finally $\h^{\alpha}(e)$ be the $\h$-module spanned by the sequence
$$Z^{\alpha}(e)=(\na^{\alpha_{\vert j}}e)_{0\lsl j\lsl\abs{\alpha}} .$$

\begin{lemm}
\label{chder}
For any $\alpha=(\alpha_1,\ldots,\alpha_{\ell})\in \Z^{\ell}$, one has 
$$\naa=z^{\norm{\alpha}}P_{\alpha}(\na_0)\textrm{ where
}P_{\alpha}(X)=X(X+\alpha_1)\cdots(X+\norm{\alpha_{\vert \ell-1}})\in \Z[X].$$ 
\end{lemm}

\begin{proof}
The proof goes by induction on the length of the multi-index $\alpha$.
Let $D=\na_0$. The claim obviously holds for a multi-index of length
0, with $P_{\epsilon}=1$, so assume that there exists $P_{\alpha}\in \Z[X]$ such that
$\naa=z^{w(\alpha)}P_{\alpha}(D)$ for $\abs{\alpha}\lsl \ell$. Let
$\beta\in \Z^{\ell+1}$. Then by definition, we have 
\begin{eqnarray*}
\nab&=& z^{\beta_{\ell+1}}D\circ\na^{\beta_{\vert \ell}}\\
           &=& z^{\beta_{\ell+1}}D\circ (z^{\norm{\beta_{\vert
           \ell}}}P_{\beta_{\vert \ell}}(D))\\
           &=& z^{\beta_{\ell+1}} \left(\norm{\beta_{\vert \ell}}
           z^{\norm{\beta_{\vert \ell}}} P_{\beta_{\vert
           \ell}}(D)+z^{\norm{\beta_{\vert \ell}}} D\circ P_{\beta_{\vert
           \ell}}(D)\right)\\
           &=& z^{\norm{\beta}} \left(\norm{\beta_{\vert \ell}} P_{\beta_{\vert
           \ell}}(D)+P_{\beta_{\vert \ell}}(D)\circ D\right)\\
           &=& z^{\norm{\beta}} \left(P_{\beta_{\vert \ell}}(D)\circ (D+\norm{\beta_{\vert \ell}})\right).
\end{eqnarray*}
Indeed, $P_{\beta_{\vert \ell}}(D)$ commutes with $D$ since it has by
assumption constant coefficients. The result follows, since we have
then $P_{\beta}(X)=P_{\beta_{\vert \ell}}(X)(X+\norm{\beta_{\vert \ell}})$.
\end{proof}

\begin{lemm}
\label{rgo}
Let $\la$ be a lattice in $V$. For any $\ell\in \N$ and $\alpha\in \N^{\ell}$, the $\h$-module
$\h^{\alpha}(e)$ is spanned over $\h$ by the family $$\left(e,z^{\alpha_1}\na_0 e,\ldots
,z^{\norm{\alpha_{\vert \ell-1}}}\na_0^{\ell-1} e\right).$$ 
\end{lemm}

\begin{proof}
According to lemma \ref{chder}, the family $Z_{\alpha}^{\ell-1}(e)$ is related to
$Z_0^{\ell-1}(e)$ by the matrix $P=Az^{W_{\alpha}}$ where
$$W_{\alpha}=\diag(0,\alpha_1,\ldots,\norm{\alpha_{\vert \ell-1}}),$$
and $A$ is an upper triangular integer matrix with diagonal entries equal to $1$,
therefore $A\in \SL_{\ell}(\Z)\subset \mathrm{GL}_{\ell}(\h)$. The families
$z^{W_{\alpha}}Z_0^{\ell-1}(e)$ and $Z^{\ell-1}_{\alpha}(e)$ are related by the matrix
$$\tP=z^{-W_{\alpha}}P=z^{-W_{\alpha}}Az^{W_{\alpha}}\textrm{ whose entries
are }A_{ij}z^{\norm{\alpha_{\vert j}}-\norm{\alpha_{\vert i}}}.$$  Since
$A$ is upper triangular, and the partial sums $\alpha_i+\cdots
+\alpha_j$ are non-negative, the matrix $\tP$ is in $\G(\h)$, and therefore both
families span the same $\h$-module.
\end{proof}

\section{Tropical convexity and lattices}
\label{ltrop}

Let $M=\{d_1,\ldots,d_m\}$ be lines in $V$ such that
$d_1+\cdots +d_m =V$, and consider the subset of $\ra$ defined by
 $$[M]=\{\ell_1+\cdots+\ell_m \tq \ell_i\textrm{ is a lattice in }d_i\}.$$
Following Keel and Tevelev, who call in \cite{KT} the set induced by
$[M]$ modulo homothety a membrane, we call this the {\em affine
  membrane spanned by $M$}.

For a choice $\A=(v_1,\ldots,v_m)$ of non-zero vectors in the lines $d_i$,
any lattice in the membrane defined by $M=\{d_1,\ldots,d_m\}$
can be represented (non uniquely) by an integer valued point as
follows: a lattice point $u\in \Z^m$ corresponds to the lattice 
\begin{equation}
\la_u=\sum_{i=1}^m \h z^{-u_i}v_i.\label{rep}
\end{equation}

Membranes spanned by $m$ lines in the Bruhat-Tits building have a faithful
representation as tropical linear spaces in $m$-dimensional space.

Let $(\TT=\R\cup\{\infty\},\oplus,\odot)$ be the tropical semialgebra,
where the operations are $$x\oplus y=\min(x,y)\textrm{ and }x\odot
y=x+y\textrm{ for }x,y\in \TT.$$
An affine membrane $M$ and a basis $(e)$ of $V$ determine a {\em valuated matroid}
$$\mapping{p}{[m]^n }{\TT}{\omega}{v(\det_{(e)} M_{\omega})}$$
where $M_{\omega}=(v_{\omega_1},\ldots,v_{\omega_n})$ is the subfamily
of vectors of $M$ indexed by $\omega$. To a valuated matroid $p$ of rank $n$ over $[m]$ there is a tropical
linear space $L_p\subset \TT^m$ attached as follows
\begin{equation}
L_p=\left\{x\in \TT^m \tq \forall \tau\in
{{[m]}\choose{n+1}},\:\min_{1\lsl i\lsl n+1}p(\tau\minus{\tau_i})+x_{\tau_i}\textrm{ is
  attained twice}\right\}.\label{lp}
\end{equation}

Depending on the authors, $L_p$ is said to be a {\it tropical convex cone}
(\cite{CGQ}) or a {\it convex polytope} (\cite{JSY}) in $\TT^m$. Both
definitions mean that 
$$\lb\odot u\oplus \mu\odot v\in L_p\textrm{ for any
}\lb,\mu\in\TT\textrm{ and }u,v\in L_p.$$ According to ~\cite{CGQ,GK,JSY}, for
$x\in \TT^m$, the formula  
\begin{equation}
\pi_{L_p}(x)=\min\{w\in L_p \tq w\gsl x\}\textrm{ where the minimum
  is taken coordinate-wise},\label{projdef}
\end{equation} defines the {\it nearest point
  projection map} $\appto{\pi_{L_p}}{\TT^m}{L_p}$. 
There are at least two other known ways to characterise or compute
$\pi_{L_p}(x)$.

\paragraph{Blue Rule.}
Adapting \cite{A}, the authors of \cite{JSY} show that
$\pi_{L_p}(x)=(w_1,\ldots,w_m)$ with
$$w_i=\min_{\sg\in{{[m]}\choose{n-1}}}\max_{j\neq
  \sg}(p(\sg\union{i})-p(\sg\union{j})+x_j).$$

\paragraph{Red Rule.}
Similarly, starting with $v=(0,\ldots,0)\in \R^m$, for every
$\tau\in{{[m]}\choose{n+1}}$ such that $\alpha=\min_{1\lsl i\lsl
  n+1}p(\tau\minus{\tau_i})+x_{\tau_i}$ is only attained once, say at
$\tau_i$, compute $\gamma=\beta-\alpha$ where $\beta$ is the second
smallest number in that collection, and put
$v_{\tau_i}:=\max(v_{\tau_i},\gamma)$. 
Then $\pi_{L_p}(x)=x+v$.

\begin{thm}[Keel-Tevelev]\label{KT}
The nearest point projection map $\appto{\pi_{L_p}}{\TT^m}{L_p}$
induces a bijection $\Psi_M$ between $[M]$ and the lattice points in
$L_p$
$$\Psi_M(\la_u) = \pi_{L_p}(u_1,\ldots,u_m).$$ 
\end{thm}

In particular,
$M_u=(z^{-u_1}v_1,\ldots,z^{-u_m}v_m)$ and
$M_{\pr{u}}=(z^{-\pr{u}_1}v_1,\ldots,z^{-\pr{u}_m}v_m)$ span the same
lattice $\la$ if and only if
$\pi_{L_p}(u)=\pi_{L_p}(\pr{u})$.

\begin{proof}
Let $\la=\sum_{i=1}^m \h z^{-u_i}v_i$, and let
$w=\pi_{L_p}(u)$. According to \cite{KT}, th. 4.17 (see also \cite{JSY} th. 18), there
exists $\alpha\in\R$ such that $w_i=v_{\la}(v_i)+\alpha$ for all $1\lsl i\lsl m$.
By definition, $v_{\la}(x)=\max\{k\in\Z\tq x\in
z^k\la\}$. Accordingly, we have $z^{-w_i}v_i\in z^{-\alpha}\la$, and
thus $\alpha \gsl 0$. By formula (\ref{projdef}), we get $\alpha=0$
and thus
\begin{equation}
\pi_{L_p}(u)=(v_{\la}(v_1),\ldots,v_{\la}(v_m)).\label{c}
\end{equation}
By construction, if $u'\gsl u$, then
$z^{-u_i}v_i=z^{(u'_i-u_i)}z^{-u'_i}v_i\in \h z^{-u'_i}v_i$, for
$1\lsl i\lsl m$, hence $\la\subset \la_{u'}$. Since in particular
$w\gsl u$ holds, we get $\la\subset \la_w$. Conversely, we have
$\la_w=\sum_{i=1}^m \h z^{-w_i}v_i\subset \la$. Therefore, if 
$\pi_{L_p}(u)=\pi_{L_p}(u')$ then $\la_u=\la_{u'}$. The converse
follows directly from (\ref{c}).
\end{proof}

\section{The G\'erard-Levelt membranes}
\label{ltrop}

\begin{prop}
\label{GLm}
Fix a basis $(e)$ of $\la$, and $\ell\gsl 0$. Let $[M_{\ell}]$ be the
  membrane spanned by the vectors $(\na_0^{j} e_i)_{1\lsl i \lsl 
  n,0\lsl j \lsl \ell}$. Then $F^{\ell'}_k(\la)\in [M_{\ell}]$ for all
  $k\gsl 0$ and $\ell'\lsl \ell$.
\end{prop}

\begin{proof}
For the considered basis $(e)$, the lattice $L=F^{\ell'}_k(\la)$ satisfies
$$L=\h^{\alpha}(e_1)+\cdots+\h^{\alpha}(e_n)\textrm{ with }\alpha=(0,k,\ldots,k\ell').$$
Reordering terms as $(e_1,\ldots,e_n,\na_0 e_1,\ldots\na_0
e_n,\ldots,\na_0^{\ell'} e_n)$, formula~(\ref{rep}) and lemma~\ref{rgo} imply that $L$
can be represented in the membrane $[M_{\ell'}]$ by the lattice point 
$$(\underbrace{0,\ldots,0}_{n\textrm{
    times}},\underbrace{-k,\ldots,-k}_{n\textrm{
    times}},\ldots,\underbrace{-k\ell',\ldots,-k\ell'}_{n\textrm{
    times}}).$$
Since by definition, $z^{-v_{\la}(v)}v\in \la$ holds for any $v\in V$,
    the module $L$ can also be represented as an element of the membrane
$[M_{\ell}]$ by $$(\underbrace{0,\ldots,0}_{n\textrm{ times}},
\underbrace{-k,\ldots,-k}_{n\textrm{ times}},\ldots,\underbrace{-k\ell',\ldots,-k\ell'}_{n\textrm{ times}},
v_{\la}(\na_0^{\ell'+1} e_1),\ldots,v_{\la}(\na_0^{\ell} e_n)).$$
\end{proof}

The lattices $F^{\ell}_k(\la)$ for $0\lsl \ell\lsl n$ can therefore all be seen
as elements of the same membrane $[M_n]$.

\begin{defn}
$\M_{\la}=[M_n]$ is called the {\emph{G\'erard-Levelt
  membrane}} attached to $\la$.
For any basis $(e)$, the lattice $F^{\ell}_k(\la)$ is represented by
the lattice point
$$u^{\ell}_k=(\underbrace{0,\ldots,0}_{n\textrm{ times}},
\underbrace{-k,\ldots,-k}_{n\textrm{ times}},\ldots,\underbrace{-k\ell,\ldots,-k\ell}_{n\textrm{ times}},
v_{\la}(\na_0^{\ell+1} e_1),\ldots,v_{\la}(\na_0^{n} e_n)).$$
\end{defn}

If $\mat(\na_0,(e))=A$ for an basis $(e)$ of $\la$, then
$\M_{\la}$ is described in $(e)$ by the $n\times n(n+1)$
matrix 
\begin{equation}
{\bf M}=\left(\begin{array}{ccccc}I_n & A & \cdots & A_n
\end{array}\right)\textrm{ where }A_{k+1}=(z\frac{d}{dz}+A)A_k\textrm{
  and }A_0=I_n.\label{iter}
\end{equation}

The tropical projection $\pi_{\la}$ onto the tropical linear space
$L_{\la}$ attached to the G\'erard-Levelt
membrane $\M_{\la}$ maps a point $u$ to a {\it unique}
representative. Checking if $k\gsl m(\na)$ requires to know if the lattice points
$u^{n-1}_k$ and $u^{n}_k$ represent the same lattice, that is
$$\pi_{\la}(u^{n}_k)=\pi_{\la}(u^{n-1}_k).$$

\begin{cor}\label{P}
For any  $\la$, we have
$m(\na)=\min\{k\in \N\tq \pi_{\la}(u^{n}_k)=\pi_{\la}(u^{n-1}_k)\}$.
\end{cor}

\subsection{Tropical computation of the Katz rank}
\label{ltrop}

The tropical setting is compatible with the ramification of the
variable. This implies the following result.

\begin{thm}
\label{K}
Let $\appto{\pi_{\la}}{\M_{\la}}{L_{\la}}$ be the tropical nearest point projection map
of the G\'erard-Levelt membrane $\M_{\la}$ of any lattice $\la$ onto its
attached tropical linear space $L$. Then the Katz rank
$\ka(\na)$ of the connection $\na$ satisfies
$$\ka(\na)=\min\{k\in\R^+\tq \pi_{\la}(u^{n}_k)=\pi_{\la}(u^{n-1}_k)\}\textrm{ for
    any lattice }\la.$$
\end{thm}

\begin{proof}
The Katz rank is the minimum Poincar\'e rank of the
connection $\na_H$ induced on the pure algebraic extension
$H=K[T]/(T^N-z)$ of $K$ with $N=\lcm(1,2,\ldots,n)$ (see {\it
  e.g.} \cite{Cor6}). If we put $\zeta$ 
for the class of $T$, then
$\mat((\na_H)_{\zeta\frac{d}{d\zeta}},(e\otimes
1))=N\mat(\na_{z\frac{d}{dz}},(e))$. Thus if $X(z)$
satisfies $z\frac{d}{dz}X(z)=A(z)X(z)$ the system satisfied by 
$Y(\zeta)=X(\zeta^N)$ is
$$\zeta\frac{d}{d\zeta}Y(\zeta)=NA(\zeta^N)Y(\zeta).$$ 
Put $\tilde{A}(\zeta)=NA(\zeta^N)$. The sequence
$(\tilde{A}_k)_{k\in\N}$ defined by relation (\ref{iter}) of
iterated $\zeta\frac{d}{d\zeta}$-derivatives of $\tilde{A}$ satisfies 
$$\tilde{A}_k(\zeta)=N^kA_k(\zeta^N).$$
Let $q$ be the valuated matroid defined by
$q(\omega)=w(\det\tilde{M}(\zeta)_{\omega})$, for any $n$-subset $\omega$ of indices of the columns of
$\tilde{M}(\zeta)$ with respect to the $\zeta$-adic valuation $w$. By construction we have
\begin{eqnarray*}
q(\omega)=w(\det\tilde{M}(\zeta)_{\omega})&=&w(\det
M(\zeta^N)_{\omega})\\
&=&w((\det M(\zeta)_{\omega})^N)\\
&=&Nw(\det M(\zeta)_{\omega})\\
&=&Np(\omega).
\end{eqnarray*}
The lattice $N_H=\sum_{i=1}^m \h_H \zeta^{-u_i}v_i\otimes 1$ has
tropical representation in $L_q$ as the projection of the point $u\in
\Z^m$ with respect to the matroid $q=Np$. 
By corollary~\ref{P},  we have 
$m(\na_H)=\min\{k\in \N\tq \pi_{N_H}(u^{n}_k)=\pi_{N_H}(u^{n-1}_k)\}$.
On the other hand, $\ka(\na)=\frac{1}{N}m(\na_H)$ holds.
Therefore, we get 
$$\ka(\na)=\min\{k\in\frac{1}{N}\N\tq \pi_{\la}(u^{n}_k)=\pi_{\la}(u^{n-1}_k)\}.$$
This formula holds for {\it any} extension $H'$ of degree divisible by
the denominator $s$ of $\ka(\na)$, hence the result also holds in the
limit, yielding the claimed result.
\end{proof}

\begin{ex}[Pfl\"ugel-Barkatou]
Let $dX/dz=AX$ with
$$
A=\left(\begin{matrix}
-5z^{-2} &5z^{-1} &-2z^{-1} &-9z^{-2} \\
5z^{-3}  &3z^{-2}  &2z^{-2} &-4z^{-2} \\
4z^{-1} &-5z^{-1} &-5z^{-2} & 2 \\
\frac{2-2z}{z^{3}} &-5z^{-1} &3z^{-2} &-6z^{-2}
\end{matrix}
\right).
$$
\begin{itemize}
\item[]$u^{n}_k=(0,\ldots,-3k,-4k-4k-4k-4k)$
\item[]$u^{n-1}_k=(0,\dots,-3k,-6,-5,-5,-6)$.
\end{itemize}
$$\textrm{One gets }\pi(u^{n}_k)=\pi(u^{n-1}_k)\iff
k\gsl\frac{3}{2},\textrm{ therefore }m(\na)=2\textrm{ but actually
}\kappa(\na)=\frac{3}{2}.$$
\end{ex}

\section{A projection algorithm on a tropical linear space}
\label{proj}
The Blue and Red rules from \cite{JSY} recalled in section \ref{ltrop} have
unfortunately a high computational complexity, since it involves loops
over cardinality $m\choose n$ sets. In our case, it is especially
impractical since for the G\'erard-Levelt membrane, we have
$m\sim n^2$. In this section, we present an efficient algorithm,
inspired by Ardila's work on ordinary matroids~\cite{A}, to
compute the projection of a point $x\in \R^m$ onto the tropical linear
space $L_p$ attached to a valuated matroid $p$.

\subsection{Valuated matroids}

Let us recall the setup of valuated matroids, and fix the notations
that we will use. For the results listed in this section, we refer to
\cite{MuTa}, although their definition, following \cite{DT}, comes
with the opposite sign. Let $E$ be a finite set, and a map
$\appto{p}{2^E}{\TT=\R\union{\infty}}$. Let $\B=\{B\subset E \tq p(B)\neq \infty\}$.
The pair $(E,p)$ is a {\it valuated matroid} if $\B\neq \emptyset$ and 
for $B,B'\in \B$ and $u\in B\backslash B'$ 
there exists $v\in B'\backslash B$ such that
$$p(B)+p(B')\gsl p(B\union{v}\minus{u})+p(B'\union{u}\minus{v}).$$

A subset $B\in \B$ is called a basis of $p$. In particular, $\B$
  is the set of bases of an ordinary matroid $P$ on 
$E$, that we call the {\it matroid underlying $p$}. Any vector of the form
$$X(B,v)=(p(B\union{v}\minus{u})-p(B),\: u\in E)$$ for some basis $B$ and $v\in
E\backslash B$ is a circuit of $p$.
If $X$ is a circuit of $p$, its {\it support} $$\ov{X}=\{e\in E \tq
X_e\neq \infty\}$$ is a circuit of the matroid $P$. More precisely, it
is the fundamental circuit of $B$ and $v$, that is the unique circuit
of $P$ included in $B\union{v}$. Similarly, any vector of the form
$$X^*(B,v)=(p(B\union{u}\minus{v})-p(B),\: u\in E)$$ for some basis $B$ and
$v\in B$ is thus a cocircuit of $p$.

Some important features of circuits and cocircuits of $p$ are in fact
encoded in the underlying matroid $P$. 
For any circuit $C$ of $P$, the set of circuits of $p$ that have
  $C$ as support is of the form $$X+\alpha(1,\ldots,1)\textrm{ for
  }\alpha\in\R.$$ Conversely, for any circuit $X$ of $p$, $X+\alpha(1,\ldots,1)$
  for $\alpha\in\R$ is a circuit of $p$. The same result applies to
  cocircuits. Recall the following result. 
\begin{lemm}
Any circuit (resp. cocircuit) of $P$ containing $v\in E$ can be represented as
the fundamental circuit (resp. cocircuit) of a basis $B$ such that
$v\notin B$ (resp. $v\in B$).
\end{lemm}

\begin{proof}
Let $C$ be a circuit of $P$. By definition, for any $v\in C$, the set
$C\minus{v}$ is contained in some basis $B$. Therefore $C\subset
B\union{v}$ holds. But there is a unique circuit satisfying this
condition. Since the cocircuits are the circuits of the dual matroid,
the same result holds.
\end{proof}

In what follows, we will speak by abuse of notation of the fundamental
(co-)circuit of $B$ and $v$ for a valuated matroid $p$. This is
harmless as long as the results that we state are invariant up to the
addition of a constant. If we need to specify a representative, we
will often use the only one with non-negative coordinates and with
minimum coordinate equal to 0, or with some fixed value at some
element of $E$.

For any $x\in\R^m$, the map
$$p_x(B)=p(B)-\sum_{b\in B}x_b$$
extended to all $2^E$ by $p_x(A)=\infty$ for $A\notin \B$ defines
another valuated matroid on $E$.
\begin{lemm}
If $X$ is any circuit of $p$, then $X+x$ is a circuit of $p_x$, and if
$X^*$ is a cocircuit of $p$, then $X^*-x$ is a cocircuit of $p_x$.
\end{lemm}

\begin{proof}
By the definition of a circuit of $p$, circuits of $p_x$ have coordinates
\begin{eqnarray*}
X_x(B,v)_u&=&p_x(B\union{v}\minus{u})-p_x(B)\textrm{ for some }B\not\ni v\\
          &=& p(\union{v}\minus{u})-\sum_{b\in B\union{v}\minus{u}}x_b-p(B)+\sum_{b\in B}x_b\\
          &=&X(B,v)_u-x_v+x_u.
\end{eqnarray*}
Hence, $X_x(B,v)=X(B,v)+x-x_v(1,\ldots,1)$.
Similarly, we have
\begin{eqnarray*}
X^*_x(B,v)_u&=&p_x(B\union{u}\minus{v})-p_x(B)\\
          &=& p(\union{u}\minus{v})-\sum_{b\in B\union{u}\minus{v}}x_b-p(B)+\sum_{b\in B}x_b\\
          &=&X^*(B,v)_u+x_v-x_u.
\end{eqnarray*}
Hence, $X^*_x(B,v)=X(B,v)-x+x_v(1,\ldots,1)$. By the projectivity
property of circuits and cocircuits, the result is established.
Since the sets of bases for $p$ and $p_x$ coincide, these are indeed
the only (co)circuits of $p_x$.
\end{proof}

\subsection{An algorithm for the projection on the tight-span}

A valuated matroid $\appto{p}{{{E}\choose{n}}}{\TT}$ of
rank $n$ over a finite set $E=[m]$ induces a tropical linear space
$L_p$ defined by~(\ref{lp}).  This subspace of $\TT^m$ corresponds (up
to sign) to what Dress and Terhalle call the {\it tight span} of a valuated matroid. In
this section, we present an efficient algorithmic method to 
compute the tropical projection from $\R^m$ onto $L_p$, that
generalises results obtained by Ardila for ordinary matroids in~\cite{A}.

\begin{prop}
\label{r}
Let $p$ be a valuated matroid of rank $n$ on $[m]$, and let $u\in E$. The following
conditions are equivalent.
\begin{enumerate}[i)]
\item $u$ belongs to at least one minimal basis of $p$.\label{ri}
\item $u$ is never the unique minimum in a circuit of $p$.\label{rii}
\item $u$ is minimal in some cocircuit of $p$.\label{riii}
\end{enumerate}
\end{prop}

\begin{proof}
$\ref{ri})\Rightarrow\ref{riii})$: Assume that $B$ is a minimal basis
containing $u$. Let $C^*=X^*(B,u)$ be the fundamental
cocircuit of $B$ and $u$. By definition, we
have $$C^*_v=p(B\union{v}\minus{u})-p(B)\gsl 0=C^*_u.$$
That is, $u$ is minimal in the cocircuit of $B$ and $u$.

$\ref{riii})\Rightarrow\ref{rii})$: suppose that $u$ in the unique
minimum for $p$ on a circuit $C$. Assume that $C^*$ is a cocircuit of $p$
where $u$ is minimal. By assumption, we have
$$C_u<C_{u'}\textrm{ and }C^*_u\lsl C^*_{u'}\textrm{ for }u'\neq u.$$
Accordingly, $C+C^*$ has a unique minimum at $u$. By orthogonality of
circuits and cocircuits (\cite{MuTa}, th. 3.11, p. 204), the set of indices that
minimise $C+C^*$ cannot have cardinality one. Therefore, the
contradiction is established.

Let us finally prove $\ref{rii})\Rightarrow \ref{ri})$: consider a minimum
basis $B$. If $u\notin B$, let $C=X(B,u)$ be the
circuit generated by $B$ and $u$. By assumption, the minimum in $C$ is
attained at $v\neq u$. The support of $C$ is equal to
the fundamental circuit of $B$ and $u$ for the ordinary matroid
underlying $p$. Therefore, $B\union{u}\minus{v}$ is a basis of $p$ and
$$p(B\union{u}\minus{v})-p(B)\lsl
p(B\union{u'}\minus{v})-p(B)\textrm{ for }u'\in C.$$
Putting $v=u'$ we get
$p(B\union{u}\minus{v})\lsl p(B)$. Since we assumed that $B$ was
minimal, we get $\ref{ri})$.
\end{proof}

Therefore we get the following characterisation of the (finite part of
the) tropical linear space $L_p$.

\begin{prop}
\label{ar}
Let $x\in\R^m$, and let $p$ be a valuated matroid of rank $n$ on $[m]$. The following are equivalent.
\begin{enumerate}[i)]
\item $x\in L_p$.\label{ari}
\item Every element of $E$ belongs at least to one $x$-minimal basis
  of $p$.\label{arii}
\item Every circuit of $p$ contains at least 2 $x$-minimal elements.\label{ariii}
\item Every element of $E$ is $x$-minimal in at least one cocircuit of $p$.\label{ariv}
\end{enumerate}
\end{prop}

\begin{proof}
$\ref{ari})$ and $\ref{ariii})$ are equivalent by the definition of $L_p$
({\it cf.} \cite{JSY}). The remaining assertions are obtained by applying
proposition~\ref{r} to the valuated matroid $p_x$.
\end{proof}

Note that the previous characterisation does not apply when $x$ has an
infinite coordinate, since $p_x$ is then no longer a valuated
matroid. However, $x_u=\infty$ happens only when $u$ does not belong
to any basis. 

The computation of $\pi_{L_p}(x)$ can be performed independently for
every element of the vector $x$. For a given $u\in E$, there is a (unique)
normalisation of a circuit $C$ of $p$ containing $u$ such that $C^x_u=x_u$.

\begin{prop}
If $u\in E$ violates any one of the three conditions of
proposition~\ref{r} for the matroid $p_x$, then $u$ satisfies them for
the modified vector $x'=(x_1,\ldots,x'_u,\ldots,x_m)$ with
$$x'_u=\max_{u\in C}\min_{e\in C\minus{u}} C^x_e,$$ where 
all the circuits are normalised so that $C^x_u=x_u$.
Moreover, the conditions of
proposition~\ref{r} are not satisfied at $u$ for
$x''=(x_1,\ldots,x'_u-\varepsilon,\ldots,x_m)$ with $\varepsilon>0$.
\end{prop}

\begin{proof}
By assumption, $u$ is the unique $x$-minimum over some circuit $\ti{C}$
containing $u$. The support of such a circuit $C$ can be defined as $\ov{C}=X(B,u)$ the
fundamental circuit of $u$ and a basis $B\not\ni u$. The $x$-value at
$e\in \ov{C}$ of the circuit $C$ is of the
form $$C^x_e=p(B\union{u}\minus{e})-p(B)+x_e+ \alpha\textrm{ for some
  constant }\alpha\in\R,$$ so we may choose as representative of any
circuit $C$ containing $u$ the only one such that
$C^x_u=x_u$, namely the one defined by
$C^x_e=p(B\union{u}\minus{e})-p(B)+x_e$.

Then by assumption $$x'_u\gsl \min_{e\in C\minus{u}}\ti{C}^{x}_e> \ti{C}^{x}_u=x_u,$$
and for any circuit $C\ni u$, we have 
$$C^{x'}_e=\left\{\begin{array}{l}C^{x'}_e\textrm{ if }e\neq u\\
                                 \\
                                 C^{x'}_u=x'_u\textrm{ if
                                 }e=u\end{array}\right.$$
Therefore $$C^{x'}_u\gsl \min_{e\in C\minus{u}}C^x_e=\min_{e\in C\minus{u}}C^{x'}_e$$
so $u$ cannot be the unique $x'$-minimum of any circuit containing $u$.
On the other hand, there exists a circuit $C$
containing $u$ such that $\min_{e\in C\minus{u}}C^{x}_e=x'_u$. Putting
$x''_u=x'_u-\varepsilon$ for any $\varepsilon>0$, then $u$ will be the
$x''$-unique minimum over the circuit $C$. Thus $x'$ is the smallest
vector that corrects the value of $x$ at $u$.
\end{proof}

\begin{prop}
Let $x\in\R^m$ and let $B$ be an $x$-minimal basis of $p$. Then
$\omega=\pi_{L_p}(x)$ can be computed in the following way.
$$\omega_i=\left\{\begin{array}{l}x_i\textrm{ if }i\in B\\
\\
\min_{u\neq i}(p(B\union{i}\minus{u})-p(B)+x_u)\textrm{ otherwise.}\end{array}\right.$$
\end{prop}

\begin{proof}
If $i \in B$ holds, then $i$ is $x$-minimal in the fundamental cocircuit
$X^*(B,i)$. Therefore all conditions of proposition~\ref{r} apply to $i$.
Otherwise, let $X(B,i)$ be the fundamental circuit of $B$ and $i$,
normalised so that $X(B,i)_i=x_i$. We have to prove that 
$$\min_{e\neq i}(X(B,i)_e+x_e)=\max_{i\in C}\min_{e\in C\minus{i}}
C^x_e.$$ By construction, $\lsl$ holds. Actually, it is sufficient to
prove that $\min_{e\neq i}(X(B,i)_e)\gsl \min_{e\in C\minus{i}}C_e$
for any circuit $C$ containing $i$.

Say that $\min_{e\neq i}X(B,i)_e=X(B,i)_u$. For any circuit
$C$ containing $i$, there exists $v\in C\minus{i}$ such that
$B\union{v}\minus{u}$ is a basis. Since $i\notin B\union{v}\minus{u}$,
the circuit $\ti{C}=X(B\union{v}\minus{u},i)$ containing $i$ is well
defined, and $\ti{C}_v\gsl \min_{e\in\ti{C}\minus{i}}\ti{C}_e$ holds.
By definition, we have 
\begin{eqnarray*}
\ti{C}_v&=&p((B\union{v}\minus{u})\union{i}\minus{v})-p(B\union{v}\minus{u})\\
        &=&p(B\union{i}\minus{u})-p(B\union{v}\minus{u})\\
        &\lsl&p(B\union{i}\minus{u})-p(B)\textrm{ since }B\textrm{ is
          minimal}\\
        &\lsl&X(B,i)_u.
\end{eqnarray*}
Accordingly, we get $X(B,i)_u\gsl \ti{C}_v\gsl \min_{e\in\ti{C}\minus{i}}\ti{C}_e$.
\end{proof}

This result implies the following efficient method to compute the
tropical projection $\pi_{L_p}(x)$ for $x\in\R^m$. 
\begin{enumerate}
\item Compute a minimal basis $B$ of $p_x$. This can be performed by
  the greedy algorithm described in~\cite{DW}.
\item For $i\in E\backslash B$, compute $C=X(B,i)$. To do this, it
  suffices to compute the fundamental circuit of $B$ and $i$ for the
  underlying matroid $P$.
\item Find the minimum element of $C+x$ outside $i$; note that there are at most $n$ non
  infinite elements to consider.
\end{enumerate}



\begin{thebibliography}{Z99}
\bibitem[Ar]{A} F. Ardila, Subdominant matroid ultrametrics, {\it
    Annals of Combinatorics}, 8, 2004, pp. 379--389.
\bibitem[B-V]{BV} D. G. Babbitt, V. S. Varadarajan, Formal reduction
  theory of meromorphic differential equations : a group theoretic
  view, {\it Pacif. J. Math.}, {\bf 109}, 1983, pp. 1--80.
\bibitem[Ba-J-L]{BJL} W. Balser, W. B. Jurkat, D. A. Lutz, A General
    Theory of Invariants for Meromorphic Differential Equations, Part I
    : Formal Invariants, {\it Funkcialaj Ekvacioj}, vol. 22, 1979,
    pp. 197--221.
\bibitem[C-G-Q]{CGQ} G. Cohen, S. Gaubert, J.P. Quadrat, Duality and
  Separation Theorems in Idempotent Semimodules, {\it Linear Alg. and
    Appl.}, Vol. 379, 2004, pp. 395--422.
\bibitem[Cor]{Cor6} E. Corel, On Fuchs' relation for linear
differential systems, {\it Compositio Math.}, {\bf 140}, 2004, pp. 1367--1398.
\bibitem[De]{De} P. Deligne, {\em \'Equations diff\'erentielles \`a
    points singuliers r\'eguliers}, Lect. Notes in Math., vol. 163,
  Springer-Verlag, 1970.
\bibitem[D-T]{DT} A. W. M. Dress, W. Terhalle, The tree of life and
  other affine buildings, {\it Documenta Mathematica}, Extra Volume
  ICM 1998, Part III, pp. 565--574.
\bibitem[D-W]{DW} A. W. M. Dress, W. Wenzel, Valuated matroid: A new
  look at the greedy algorithm, {\it Appl. Math. Letters}, 3, 1990,
  pp. 33-35.
\bibitem[Fu]{Fu} L. I. Fuchs, Zur Theorie der linearen
Differentialgleichungen mit ver\"anderlichen Coeffizienten, {\em
J. rein. angew. Math.} {\bf 66}, 1866, pp. 121--160.
\bibitem[G-K]{GK} S. Gaubert, R. Katz, Minimal half-spaces and
  external representation of tropical polyhedra, {\it
    J. Alg. Combinatorics}, vol. 33, n$^{\circ}$3, 2011, pp. 325--348.
\bibitem[G-L]{GL} R. G\'erard et A. H. M. Levelt, Invariants
mesurant l'irr\'egularit\'e en un point singulier d'un syst\`eme
d'\'equations diff\'erentielles lin\'eaires, {\it Ann. Inst.
Fourier}, Grenoble, {\bf 23} (1), 1973, pp. 157--195.
\bibitem[Ju]{Ju} W. Jurkat, {\it Meromorphe Differentialgleichungen},
    Lect. Notes in Math. 637, Springer Verlag, 1978.
\bibitem[Ka]{Ka} N. Katz, Nilpotent connections and the
    monodromy theorem. Applications of a result of Turrittin,
    {\it Publ. Math. IHES}, {\bf 39}, 1970, pp. 176--232.
\bibitem[H-W]{HW} A. Hilali, A. Wazner, Formes super-irr\'eductibles
des syst\`emes diff\'erentiels lin\'eaires, {\em Num. Math.} {\bf 50}
(4), 1987, pp. 429--449.
\bibitem[J-S-Y]{JSY} M. Joswig, B. Sturmfels, J. Yu, Affine buildings and tropical geometry,
{\it Alb. J. Math.}, 4, pp 187--211.
\bibitem[K-T]{KT} S. Keel, J. Tevelev, Geometry of Chow quotiens of Grassmannians,
{\it Duke Math. J.} {\bf 134}, no. 2, 2006, pp. 259--311.
\bibitem[M]{M} Y. Manin, Moduli fuchsiani,
    {\it Ann. Sc. Norm. Sup. Pisa}, {\bf{19}} (1) Serie III, 1965,
    pp. 113--126.
\bibitem[Mo]{Mo} J. Moser, The order of a singularity in Fuchs' theory,
    {\em Math. Z.} {\bf 72}, 1960, pp. 379--398.
\bibitem[Mu-Ta]{MuTa} K. Murota, A. Tamura, On circuit valuation of
  matroids, {\it Adv. Appl. Math.}, 26, n$^{\circ}3$, 2001, pp. 192--225.
\bibitem[W]{W}A. Werner, A tropical view on Bruhat-Tits buildings and
  their compactifications, {\it Cent. Eur. J. Math.} 9, n$^{\circ}$2,
  2011, pp. 390--402.
\bibitem[Y-Y]{YY} J. Yu, D. S. Yuster, Representing tropical linear spaces
by circuits, {\it Proceedings of FPSAC}, 2007.
\end{thebibliography}
\end{document}